%% file: rspider.tex
\documentclass{article}

\usepackage{arxiv}
\usepackage{natbib}
\usepackage[utf8]{inputenc} 
\usepackage[T1]{fontenc}    
\usepackage{hyperref}       
\usepackage{url}            
\usepackage{booktabs}       
\usepackage{amsfonts}       
\usepackage{nicefrac}       
\usepackage{microtype}      
\usepackage{makecell}

\usepackage{amssymb}
\usepackage{mathtools}
\usepackage{amsthm}
\usepackage{color}
\usepackage{caption} 
\usepackage{algorithmic}
\usepackage{algorithm}

\definecolor{darkblue}{rgb}{0.0,0.0,0.55}
\hypersetup{
  colorlinks = true,
  citecolor  = darkblue,
  linkcolor  = darkblue,
  citecolor  = darkblue,
  filecolor  = darkblue,
  urlcolor   = darkblue,
}

\usepackage{enumitem}
\usepackage{xspace}

\usepackage{wrapfig}

\newtheorem{lemma}{Lemma}

\newtheorem{assumption}{Assumption}
\newtheorem{theorem}{Theorem}

\newtheorem{definition}{Definition}

\newcommand{\cS}{\mathcal{S}}
\newcommand{\E}{\mathbb{E}}
\newcommand{\F}{\mathcal{F}}
\newcommand{\Oc}{\mathcal{O}}
\newcommand{\Exp}{\mathrm{Exp}}
\newcommand{\Mc}{\mathcal{M}}
\newcommand{\reals}{\mathbb{R}}
\newcommand{\grad}{\nabla}
\newcommand{\x}{x}
\newcommand{\algo}{\textsc{R-Spider}\xspace}

\newcommand{\inner}[1]{\langle #1 \rangle}
\DeclarePairedDelimiter\floor{\lfloor}{\rfloor}
\DeclarePairedDelimiter\ceil{\lceil}{\rceil}
\title{R-SPIDER: A Fast Riemannian Stochastic Optimization Algorithm with Curvature Independent Rate}

\author{\name Jingzhao Zhang \email{jzhzhang@mit.edu}\\
  \name Hongyi Zhang \email{hongyiz@mit.edu}\\
  \name Suvrit Sra \email{suvrit@mit.edu}\\
  \addr{Massachusetts Institute of Technology, Cambridge, MA}
}

\begin{document}
\maketitle

\begin{abstract}
  \input{abstract}
\end{abstract}

\section{Introduction} \label{sec:introduction}

  \input{introduction}

\section{Preliminaries} \label{sec:preliminaries}

  \input{background}

\section{Riemannian SPIDER} \label{sec:main}

\input{nonconvex}

  \input{graddom}


\section{Discussion} \label{sec:discussion}

  \input{discussion}

\bibliographystyle{abbrvnat}
\bibliography{rsvrg}

\vspace*{-5pt}

\clearpage
\setcounter{page}{1}
\appendix
\input{appendix}

\end{document}

%% file: abstract.tex
We study smooth stochastic optimization problems on Riemannian manifolds. Via adapting the recently proposed SPIDER algorithm \citep{fang2018spider} (a variance reduced stochastic method) to Riemannian manifold, we can achieve faster rate than known algorithms in both the finite sum and stochastic settings.  Unlike previous works, by \emph{not} resorting to bounding iterate distances, our analysis yields curvature independent convergence rates for both the nonconvex and strongly convex cases.


%% file: introduction.tex
We analyze fast stochastic algorithms for the following optimization problem:
\begin{equation} \label{eq:prob-def}
	\min_{x \in \Mc} \ f(x)\ \triangleq\ \E_\xi[f(x; \xi)],
\end{equation}
where $(\mathcal{M},\mathfrak{g})$ is a Riemannian manifold equipped with the metric $\mathfrak{g}$, and $\xi$ is a random variable. We assume that for any $\xi$, the function $f(\cdot; \xi): \mathcal{M}\to\mathbb{R}$ is geodesically $L$-smooth (see Section \ref{sec:functions}). This class of functions includes as special cases important problems such as principal component analysis (PCA), independent component analysis (ICA), dictionary learning, mixture modeling, among others. Moreover, the finite-sum problem ($f(x) = \frac{1}{n}\sum_{i=1}^n f_i(x)$) is a special case in which finite number of component functions are chosen uniformly at random (e.g., in Empirical Risk Minimization).

When solving problems with parameters constrained to lie on a manifold, a naive  approach is to alternate between optimizing the cost in a suitable ambient Euclidean space and ``projecting'' onto the manifold. For example, two well-known methods to compute the leading eigenvector of symmetric matrices, power iteration and Oja's algorithm~\citep{oja1992principal}, are in essence projected gradient and projected stochastic gradient algorithms. For certain manifolds (e.g., positive definite matrices), projections can be quite expensive to compute and possibly the Euclidean approach may have poor numerical conditioning~\citep{YHAG2017}.

An effective alternative is to use \emph{Riemannian optimization}, which directly operates on the manifold in question. This mode of operation allows Riemannian optimization to view the constrained optimization problem~\eqref{eq:prob-def} as an unconstrained problem on a manifold, and thus, to be ``projection-free.'' More important is the conceptual viewpoint: by casting the problem in a Riemannian framework, one can discover insights into problem geometry that can translate into not only more precise mathematical analysis but also more efficient optimization algorithms.

The Euclidean version of~\eqref{eq:prob-def} where $\Mc=\reals^d$ and $\mathfrak{g}$ is the Euclidean inner-product has been the subject of intense algorithmic development in machine learning and optimization, starting with the classical work of~\citet{robmon51}. However, both batch and stochastic gradient methods suffer from high computation load. For solving finite sum problems with $n$ components, the full-gradient method requires $n$ derivatives at each step; the stochastic method requires only one derivative, but at the expense of slower $O(\tfrac{1}{\epsilon^2})$ convergence to an $\epsilon$-accurate solution. These issues have motivated much of the progress on faster stochastic optimization in vector spaces by using variance reduction~\citep{schmidt2013minimizing,johnson2013accelerating,defazio2014saga,konevcny2013semi}. Along with many recent works (see related work), these algorithms achieve faster convergence than the original gradient descent algorithms in multiple settings.

Riemannian counterparts of batch and stochastic optimization algorithms have witnessed growing interest recently. \citet{zhang2016first} present the first global complexity analysis of batch and stochastic gradient methods for geodesically convex functions. Later work \citep{zhang2016riemannian,kasai2016riemannian,sato2017riemannian} improves the convergence rate for finite-sum problems by using variance reduction techniques. In this paper, we develop this line of work further, and improve rates based on a more careful control of variance analyzed in~\citep{fang2018spider, nguyen2017sarah}. Another important aspect of our work is that by pursuing a slightly different analysis, we are able to remove the assumption that all iterates remain in a compact subset of the Riemannian manifold. Such an assumption was crucial to most previous Riemannian methods, but was not always fully justified.

\paragraph{Contributions.}
We summarize the key contributions of this paper below. 

\begin{itemize}[leftmargin=*]
  \setlength{\itemsep}{0pt}
\item We introduce \algo, a Riemannian variance reduced stochastic gradient method based on the recent SPIDER algorithm~\citep{fang2018spider}. We analyze \algo for optimizing geodesically smooth stochastic nonconvex functions. To our knowledge, we obtain the first rate faster than Riemannian stochastic gradient descent for general nonconvex stochastic Riemannian optimization. 

\item We specialize \algo to (Riemannian) nonconvex finite-sum problems. Our rate improves the best known rates and match the lower bound as in the Euclidean case. 

\item We propose two variations of \algo for geodesically strongly convex problems and for Riemannian gradient dominated costs. For these settings, we achieve the best known rates in terms of number of samples $n$ and the condition  number $\kappa$.

\item Importantly, we provide convergence guarantees that are \emph{independent of the Riemannian manifold's diameter and its sectional curvature}. This contribution is important in two main aspects. First, the best known theoretical upper bounds are improved. Second, the algorithm no longer assumes bounded diameter of the Riemannian manifold, which helps lift the assumption crucial for previous work that required all the iterates generated by the algorithm to remain in a compact set.
\end{itemize}

We briefly summarize the rates obtained in Table~\ref{tab:comparisons}. 

\begin{table}\centering \label{tab:comparisons}
	\setlength{\tabcolsep}{3pt}
	\renewcommand{\arraystretch}{2}
	\begin{tabular*}{1.0\textwidth}{c|c|c|c|c}
		
		\hline
		      & \makecell{Nonconvex \\stochastic} &\makecell{ Nonconvex \\ finite sum} & \makecell{Strongly convex \\ finite sum} & \makecell{Gradient dominated \\ finite sum}  \\
		\hline
		\makecell{Previous work } & $\Oc(\frac{1}{\epsilon^4})$ & $\Oc(n + \frac{n^{2/3}\zeta^{1/2}}{\epsilon^2})$ & $\Oc((n + \kappa^2\zeta)\log(\tfrac{1}{\epsilon}))$ & $\Oc((n + n^{2/3}\zeta^{1/2}\kappa)\log(\tfrac{1}{\epsilon}))$\\ 
		\hline
		Our work  & $\Oc(\frac{1}{\epsilon^3})$ & $\Oc(n + \frac{n^{1/2}}{\epsilon^2})$ & \multicolumn{2}{c}{$\min\{\Oc((n + \kappa^2)\log(\tfrac{1}{\epsilon})), \Oc((n + \kappa n^{1/2})\log(\tfrac{1}{\epsilon}))\}$}\\ 
		\hline

	\end{tabular*}
	
	\caption{IFO complexity (of ensuring $\E[\|\nabla f\|^2] \le O(\epsilon^2)$) comparison between our work and previous works~\citep{zhang2016riemannian,sato2017riemannian,kasai2016riemannian}. The condition number $\kappa = \tfrac{L}{\mu}$ of  a $L-$smooth, $\mu-$strongly convex function or $\kappa = 2L\tau$ for a $L-$smooth, $\tau-$gradient dominated function; $\zeta$ is a constant determined by the manifold curvature and diameter. Please see \citep{zhang2016riemannian} for more details.}\label{tab:comparisons}
\end{table}

\subsection{Related Work}
\paragraph{Variance reduction in stochastic optimization.} Variance reduction techniques, such as \emph{control variates}, are widely used in Monte Carlo simulations \citep{rubinstein2011simulation}. In linear spaces, variance reduced methods for solving finite-sum problems have recently witnessed a huge surge of interest \citep[e.g.][]{schmidt2013minimizing,johnson2013accelerating,defazio2014saga,bach2013non,konevcny2013semi,xiao2014proximal,gong2014linear}. They have been shown to accelerate finite sum optimization for strongly convex objectives and convex objectives. Later work by \citet{lin2015universal,allen2017katyusha} further accelerates the rates in convex problems using techniques similar to Nesterov's acceleration method~\citep{nesterov2013introductory}. For nonconvex problems, \citet{reddi2016stochastic,allen2017natasha,lei2017non,fang2018spider, nguyen2017sarah} also achieved faster rate than the vanilla (stochastic) gradient descent method in both finite sum settings and stochastic settings. Our analysis is inspired mainly by \citet{fang2018spider,zhang2016riemannian}. The analysis can also be applied to \citep{wang2018spiderboost} and achieve matching rate assuming access to proximal oracle on Riemannian manifold.


\paragraph{Riemannian optimization.} Earlier references can be found in \cite{udriste1994convex,absil2009optimization}, where analysis is limited to asymptotic convergence (except \cite[Theorem 4.2]{udriste1994convex}).
Stochastic Riemannian optimization has been previously considered in~\cite{bonnabel2013stochastic,liu2004}, though with only asymptotic convergence analysis, and without any rates.
Many applications of Riemannian optimization are known, including matrix factorization on fixed-rank manifold \citep{vandereycken2013low,tan2014riemannian}, dictionary learning \citep{cherian2015riemannian,sun2015complete}, optimization under orthogonality constraints \citep{edelman1998geometry,moakher2002means}, covariance estimation \citep{wiesel2012geodesic}, learning elliptical  distributions \citep{zhang2013multivariate,sra2013geometric}, Poincar\'e embeddings \citep{nickel2017poincare} and Gaussian mixture models \citep{hoSr15b}. \citet{zhang2016first} provide the first global complexity analysis for first-order Riemannian algorithms, but their analysis is restricted to geodesically convex problems with full or stochastic gradients. \citet{boumal2016global} analyzed iteration complexity of Riemannian trust-region methods, whereas \citet{bento2017iteration} studied non-asymptotic convergence of Riemannian gradient, subgradient and proximal point methods. \citet{tripuraneni2018averaging,zhang2018towards} analyzed aspects other than variance reduction to accelerate the convergence of first order optimization methods on Riemannian manifolds. \cite{zhang2016riemannian,sato2017riemannian} analyzed variance reduction techniques on Riemannian manifolds, and their rate has remain best-known up to our knowledge. In this work, we improve upon their results. \cite{zhou2018faster} worked on the same problem in parallel and achieved the same rate. The difference between this work and \citet{zhou2018faster} is that our algorithm uses a constant step size and adaptive sample size. This enables us to bound $\E[\|\grad f(x)\|^2]$ instead of $\E[\|\grad f(x)\|]$. Hence our result is slightly stronger and further simplifies later proof for the convergence of gradient dominated functions.


%% file: background.tex
Before formally discussing Riemannian optimization, let us recall some foundational concepts of Riemannian geometry. For a thorough review one can refer to any classic text, e.g.,\citep{petersen2006riemannian}.

A \emph{Riemannian manifold} $(\mathcal{M}, \mathfrak{g})$ is a real smooth manifold $\mathcal{M}$ equipped with a Riemannain metric $\mathfrak{g}$. The metric $\mathfrak{g}$ induces an inner product structure in each tangent space $T_x\mathcal{M}$ associated with every $x\in\mathcal{M}$.  We denote the inner product of $u,v\in T_x\mathcal{M}$ as $\langle u, v \rangle \triangleq \mathfrak{g}_x(u,v)$; and the norm of $u\in T_x\mathcal{M}$ is defined as $\|u\| \triangleq \sqrt{\mathfrak{g}_x(u,u)}$. The angle between $u,v$ is defined as $\arccos\frac{\langle u, v \rangle}{\|u\|\|v\|}$. A geodesic is a constant speed curve $\gamma: [0,1]\to\mathcal{M}$ that is locally distance minimizing. An exponential map $\Exp_x:T_x\mathcal{M}\to\mathcal{M}$ maps $v$ in $T_x\mathcal{M}$ to $y$ on $\mathcal{M}$, such that there is a geodesic $\gamma$ with $\gamma(0) = x, \gamma(1) = y$ and $\dot{\gamma}(0) \triangleq \frac{d}{dt}\gamma(0) = v$.  If between any two points in $\mathcal{X}\subset\mathcal{M}$ there is a unique geodesic, the exponential map has an inverse $\Exp_x^{-1}:\mathcal{X}\to T_x\mathcal{M}$ and the geodesic is the unique shortest path with $\|\Exp_x^{-1}(y)\| = \|\Exp_y^{-1}(x)\|$ the geodesic distance between $x,y\in\mathcal{X}$.

Parallel transport $\Gamma_x^y: T_x\mathcal{M}\to T_y\mathcal{M}$ maps a vector $v\in T_x\mathcal{M}$ to $\Gamma_x^y v\in T_y\mathcal{M}$, while preserving norm, and roughly speaking, ``direction,'' analogous to translation in $\mathbb{R}^d$. A tangent vector of a geodesic $\gamma$ remains tangent if parallel transported along $\gamma$. Parallel transport preserves inner products. 

\begin{figure}[hbt]
	\centering \def\svgwidth{130pt}
	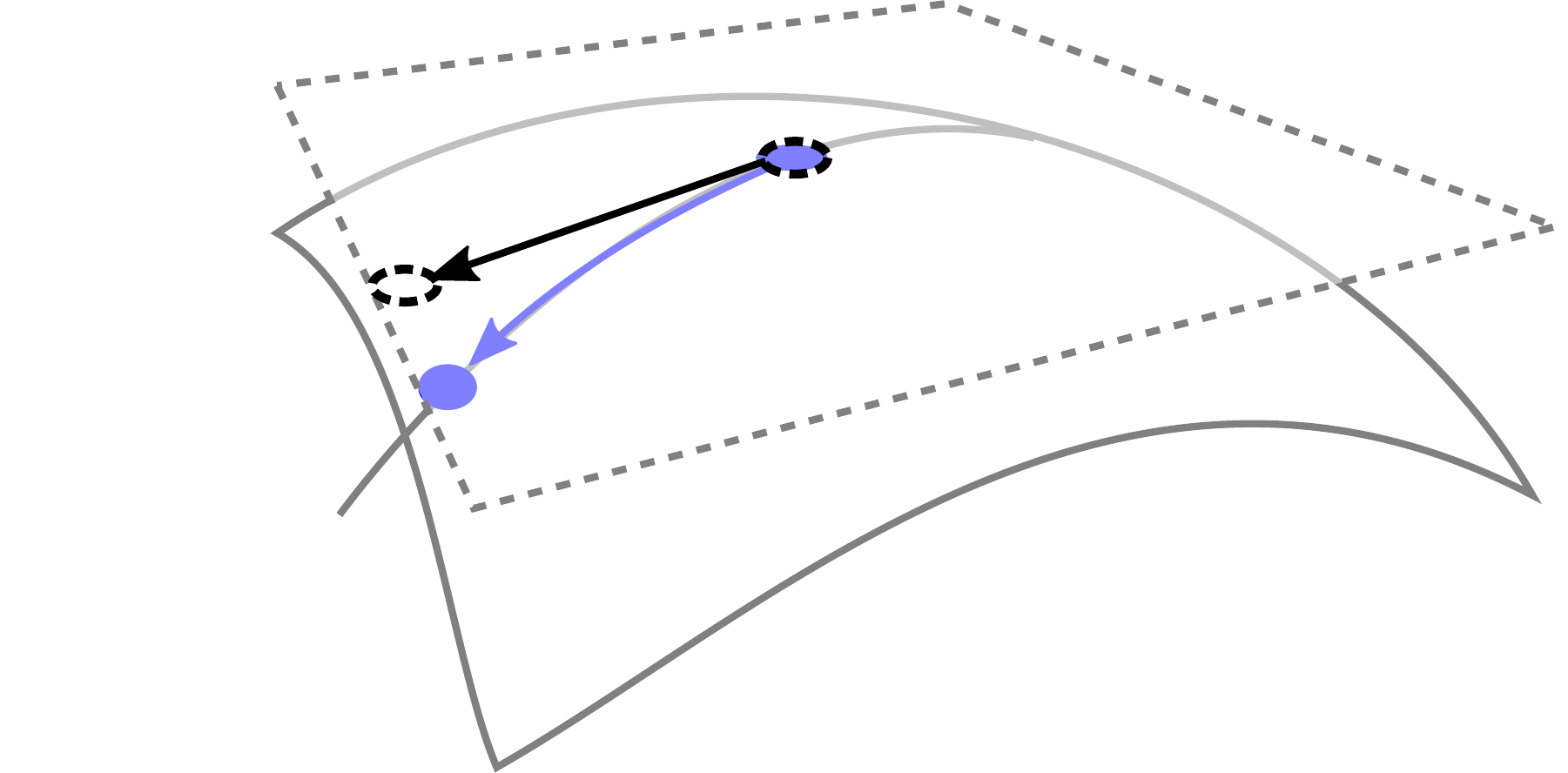 \hspace{50pt} \def\svgwidth{120pt}
	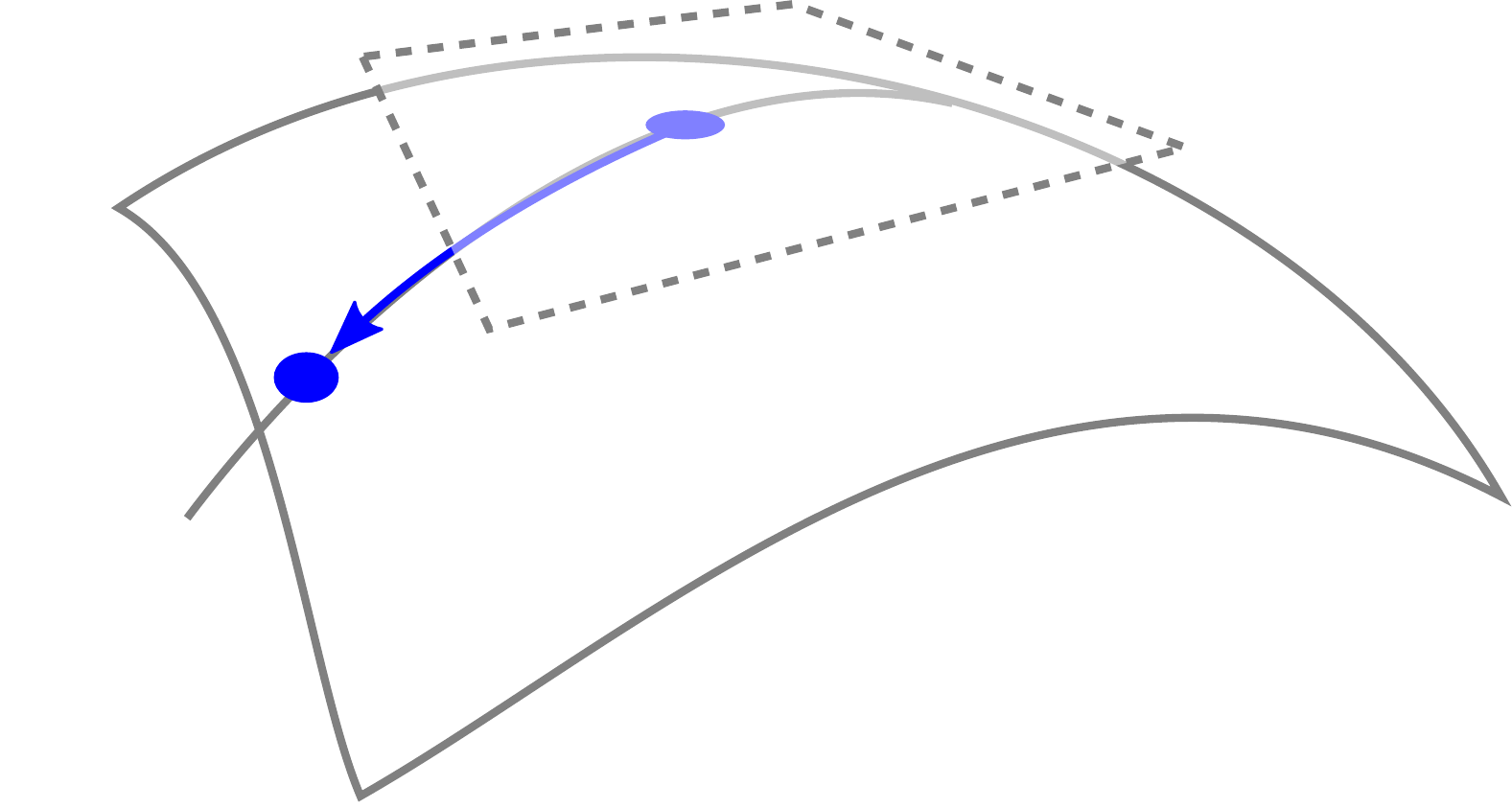
	\caption{\small Illustration of manifold operations. (Left) A vector $v$ in $T_x\mathcal{M}$ is mapped to $\Exp_x(v)$; (right) A vector $v$ in $T_x\mathcal{M}$ is parallel transported to $T_y\mathcal{M}$ as $\Gamma_x^y v$.}
\end{figure}

\paragraph{Function Classes.}
\label{sec:functions}
We now define some key terms. A set $\mathcal{X}$ is called \emph{geodesically convex} if for any $x,y\in\mathcal{X}$, there is a geodesic $\gamma$ with $\gamma(0) = x, \gamma(1) = y$ and $\gamma(t)\in\mathcal{X}$ for $t\in [0,1]$. Throughout the paper, we assume that the function $f$ in~(\ref{eq:prob-def}) is defined on a Riemannian manifold $\mathcal{M}$. 

In the following we do not explicitly write Riemannian metric $\mathfrak{g}$ or the index $x$ of tangent space $T_x\mathcal{M}$ to simplify notation, as they should be obvious from the context: inner product of $u,v\in T_x\mathcal{M}$ is defined as $\langle u, v \rangle \triangleq \mathfrak{g}_x(u,v)$; norm of $u\in T\mathcal{M}_x$ is defined as $\|u\| \triangleq \sqrt{\mathfrak{g}_x(u,u)}$. 
	
Based on the above notations, we define the following properties of the function $f$ in \eqref{eq:prob-def}.
	\begin{definition}[Strong convexity]
		A function $f:\mathcal{M}\to\mathbb{R}$ is said to be geodesically $\mu$-strongly convex if for any $x,y\in\mathcal{M}$,
		\[ f(y) \ge f(x) + \langle g_x, \Exp_x^{-1}(y) \rangle_x + \frac{\mu}{2}\|\Exp_x^{-1}(y)\|^2.\]
	\end{definition}
	\begin{definition}[Smoothness]
		A differentiable function $f:\mathcal{M}\to\mathbb{R}$ is said to be geodesically $L$-smooth if its gradient is $L$-Lipschitz, i.e. for any $x,y\in\mathcal{M}$,
		\[ \|g_x - \Gamma_y^x g_y\| \le L\|\Exp_x^{-1}(y)\|, \]
		where $\Gamma_y^x$ is the parallel transport from $y$ to $x$.
	\end{definition}
	Observe that compared to the Euclidean setup, the above definition requires a parallel transport operation to ``transport'' $g_y$ to $g_x$. It can be proved that if $f$ is $L$-smooth, then for any $x,y\in\mathcal{M}$,
	\begin{align}\label{eq:l-smooth}
		f(y) \le f(x) + \langle g_x, \Exp_x^{-1}(y) \rangle_x + \frac{L}{2}\|\Exp_x^{-1}(y)\|^2. 
	\end{align}

	\begin{definition}[PL inequality]
	$f:\mathcal{X}\to\mathbb{R}$ is \emph{$\tau$-gradient dominated} if $x^*$ is a global minimizer of $f$ and for every $x\in\mathcal{X}$
	\begin{equation} \label{eq:gradient-dominated}
	f(x) - f(x^*) \le \tau \|\nabla f(x)\|^2.
	\end{equation}
	\end{definition}
     As in the Euclidean case, $\tau-$gradient dominated is implied by $\tfrac{1}{2\tau}-$strongly convex.

An \emph{Incremental First-order Oracle (IFO)} \citep{agarwal2015lower} in (\ref{eq:prob-def})  takes in a point $x\in\Mc$, and generates a random sample $\xi$. The oracle then returns a pair $(f(x; \xi), \nabla f(x;\xi))\in \mathbb{R} \times T_x\mathcal{M}$. In finite-sum setting, $\xi$ takes values in $\{1, 2, ..., n\}$ and each random sample $f(\cdot; \xi)$ corresponds to one of $n$ component functions. We measure non-asymptotic complexity in terms of IFO calls.


%% file: figures/manifold.pdf_tex
\begingroup%
  \makeatletter%
  \providecommand\color[2][]{%
    \errmessage{(Inkscape) Color is used for the text in Inkscape, but the package 'color.sty' is not loaded}%
    \renewcommand\color[2][]{}%
  }%
  \providecommand\transparent[1]{%
    \errmessage{(Inkscape) Transparency is used (non-zero) for the text in Inkscape, but the package 'transparent.sty' is not loaded}%
    \renewcommand\transparent[1]{}%
  }%
  \providecommand\rotatebox[2]{#2}%
  \ifx\svgwidth\undefined%
    \setlength{\unitlength}{518.57113353bp}%
    \ifx\svgscale\undefined%
      \relax%
    \else%
      \setlength{\unitlength}{\unitlength * \real{\svgscale}}%
    \fi%
  \else%
    \setlength{\unitlength}{\svgwidth}%
  \fi%
  \global\let\svgwidth\undefined%
  \global\let\svgscale\undefined%
  \makeatother%
  \begin{picture}(1,0.49284658)%
    \put(0,0){\includegraphics[width=\unitlength,page=1]{figures/manifold.pdf}}%
    \put(0.53433516,0.33575391){\color[rgb]{0,0,0}\makebox(0,0)[lb]{\smash{$x$}}}%
    \put(0.33778866,0.35248126){\color[rgb]{0,0,0}\makebox(0,0)[lb]{\smash{$v$}}}%
    \put(-0.00512225,0.21448048){\color[rgb]{0,0,1}\makebox(0,0)[lb]{\smash{$\mathrm{Exp}_x(v)$}}}%
  \end{picture}%
\endgroup%

%% file: figures/parallel-transport.pdf_tex
\begingroup%
  \makeatletter%
  \providecommand\color[2][]{%
    \errmessage{(Inkscape) Color is used for the text in Inkscape, but the package 'color.sty' is not loaded}%
    \renewcommand\color[2][]{}%
  }%
  \providecommand\transparent[1]{%
    \errmessage{(Inkscape) Transparency is used (non-zero) for the text in Inkscape, but the package 'transparent.sty' is not loaded}%
    \renewcommand\transparent[1]{}%
  }%
  \providecommand\rotatebox[2]{#2}%
  \ifx\svgwidth\undefined%
    \setlength{\unitlength}{453.97390834bp}%
    \ifx\svgscale\undefined%
      \relax%
    \else%
      \setlength{\unitlength}{\unitlength * \real{\svgscale}}%
    \fi%
  \else%
    \setlength{\unitlength}{\svgwidth}%
  \fi%
  \global\let\svgwidth\undefined%
  \global\let\svgscale\undefined%
  \makeatother%
  \begin{picture}(1,0.53452335)%
    \put(0,0){\includegraphics[width=\unitlength,page=1]{figures/parallel-transport.pdf}}%
    \put(0.47731933,0.41844317){\color[rgb]{0,0,0}\makebox(0,0)[lb]{\smash{$x$}}}%
    \put(0,0){\includegraphics[width=\unitlength,page=2]{figures/parallel-transport.pdf}}%
    \put(0.35188635,0.33952307){\color[rgb]{0,0,0}\makebox(0,0)[lb]{\smash{$v$}}}%
    \put(0.22291058,0.23204309){\color[rgb]{0,0,0}\makebox(0,0)[lb]{\smash{$y$}}}%
    \put(-0.00585111,0.13126577){\color[rgb]{0,0,1}\makebox(0,0)[lb]{\smash{$\Gamma_x^y v$}}}%
  \end{picture}%
\endgroup%

%% file: nonconvex.tex
In this section we introduce the R-SPIDER algorithm. In particular, we propose one variant for nonconvex problems, and two for gradient-dominated problems. Each variation aims to optimize a particular dependency on function parameters. Our proposed algorithm differs from the Euclidean SPIDER~\citep{fang2018spider} in two key aspects: the variance reduction step uses parallel transport to combine gradients from different tangent spaces; and the exponential map is used (instead of the update $x_{k}-\eta v_{k}$).

We would like to point out that if retractions instead of exponential maps are used in the proposed algorithms, our analysis will still hold if $\exists \lambda > 0$ such that $\forall v \in T_x\Mc$,  $d(\Exp_{x}(v),  \text{Retr}_x(v)) \le \lambda\|v\|^2$, where $d(x,y) = \|\Exp_x^{-1}(y)\|$.

\subsection{General smooth nonconvex functions}
 \begin{algorithm}[ht]
 \caption{R-SPIDER-nonconvex($x_0, \cS_1,  q, \eta_k, \epsilon, f, T$)}\label{algo:nonconvex}
 \begin{algorithmic}
 	\FOR{$t = 0, 1, \ldots T$}
 	\IF{$\mod(t, q) = 0$}
 	\STATE draw $\cS_1$ samples
 	\STATE $v_k \leftarrow \grad f_{\cS_1} (x_k)$ 
 	\ELSE
 	\STATE draw $\cS_2 = \ceil{\min\{n, \frac{qL^2\|Exp^{-1}_{x_{k-1}}(x_k)\|^2}{2\epsilon^2}\}}$ samples ($n=\infty$ in the stochastic setting.)
 	\STATE $v_k \leftarrow \grad f_{\cS_2} (x_k) - \Gamma_{x_{k-1}}^{x_k} [ \grad f_{\cS_2} (x_{k-1}) - v_{k-1}]$
 	\ENDIF
	\STATE $x_{k+1} \leftarrow \Exp_{x_k} ( - \eta_k v_k)$
 	\ENDFOR
	\RETURN uniformly randomly from $\{x_1, ..., x_T\}$.
 \end{algorithmic}
\end{algorithm}

 Our proposed algorithm for solving nonconvex Riemannian optimization problems is shown in Algorithm~\ref{algo:nonconvex}. $\grad_\cS f(x)$ denotes the \emph{unbiased} gradient estimator obtained by averaging $\cS$ samples. We first analyze the global complexity for solving nonconvex stochastic Riemannian optimization problems. In particular, we make the following assumptions
 
 \begin{assumption}[Smoothness]\label{assump:smooth-xi}
	For any fixed $ \xi, f(x;\xi)$ is geodesically $L-$smooth in $x$.
\end{assumption}

  \begin{assumption}[Bounded objective]\label{assump:M}
 	Function $f$ is bounded below. Define $M := f(x_0) - f^* \le \infty$ where $f^* = \inf_{x\in \Mc} f(x).$
 \end{assumption}
 
 \begin{assumption}[Bounded variance]\label{assump:variance}
	$\forall x, \E_\xi[ \|\grad f(x) - \grad f(x; \xi)\|^2] \le \sigma^2.$
 \end{assumption}
 
 Under the above assumptions, we make the following choice of parameters for running Algorithm ~\ref{algo:nonconvex}.
 \begin{align}\label{eq:non-sto-param}
 \cS_1 = 2\sigma^2/\epsilon^2, \quad \eta_k = \frac{1}{2L} ,\quad q = 1/\epsilon, \quad T = 4ML/\epsilon^2.
 \end{align}
 We now state the following theorem for optimizing stochastic nonconvex functions.
\begin{theorem}[Stochastic objective] \label{thm:stochastic}
	Under Assumptions \ref{assump:smooth-xi}, \ref{assump:M}, \ref{assump:variance} and the parameter choice in \eqref{eq:non-sto-param}, Algorithm \ref{algo:nonconvex} terminates in $4ML/\epsilon^2$ iterations. The output $x$ satisfies
	$$\mathbb{E}[\|\grad f(x)\|^2] \le 10\epsilon^2.$$
	Furthermore, the algorithm makes less than $8ML(\sigma^2+3)/\epsilon^3 $ IFO calls in expectation.
\end{theorem}
\begin{proof}
 See Appendix \ref{sec:proof-stochastic}. The gist of our proof is to show that with sufficiently small variance of the gradient estimate, the algorithm either substantially decreases the objective function every $q$ iterations, or terminates because the gradient norm is small.
\end{proof}

Then we study the nonconvex problem under the finite-sum setting. In this setting, we assume $\xi \sim \text{Uniform}(\{1,..,n\})$. Hence we can write
\begin{align}
	f(x) = \frac{1}{n}\sum_{i=1}^n f_i(x).
\end{align}
We further make the following choice of parameters:
\begin{align}\label{eq:non-finite-param}
\cS_1 = n,\quad \eta_k =\frac{1}{2L},\quad q = \ceil{n^{1/2}}, \quad T = 4ML/\epsilon^2
\end{align}
Then we have the following gurrantee.
\begin{theorem}[Finite-sum objective] \label{thm:finite-nonconvex}
	Under Assumptions \ref{assump:smooth-xi}, \ref{assump:M} and the parameter choice in \eqref{eq:non-finite-param}, Algorithm \ref{algo:nonconvex} terminates in $4ML/\epsilon^2$ iterations. The output $x$ satisfies
	$$\mathbb{E}[\|\grad f(x)\|^2] \le 10\epsilon^2.$$
	Furthermore, the algorithm makes less than $n + \frac{8ML(3+n^{1/2})}{\epsilon^2}$ IFO calls in expectation.
\end{theorem}
\begin{proof}
	See Appendix \ref{sec:proof-finite-nonconvex}. The proof is almost the same as the proof for Theorem~\ref{thm:stochastic}.
\end{proof}

The proof of the two theorems in this section follows by carefully applying the variance reduction technique proposed in \cite{fang2018spider} onto the Riemannian manifold using the tools in \cite{zhang2016riemannian}. Unlike the SVRG-like algorithms in \cite{zhang2016riemannian,kasai2016riemannian}, we can avoid analyzing the term $\Exp_{\tilde{x}}^{-1}(x_k)$, where $\tilde{x}$ is the snapshot point. Consequently, we do not need to resort to the trigonometric distance bound (see \cite{zhang2016first}) and the convergence rate doesn't depend on the sectional curvature bound.

Further, the convergence rates in both cases match their Euclidean counterparts. Remarkably, the $\Oc(n + \frac{n^{1/2}}{\epsilon^2})$ rate under the finite-sum setting meets the lower bound as proved by \citet{fang2018spider}.

\setlength{\textfloatsep}{0pt}

%% file: graddom.tex
\subsection{Gradient-dominated functions}

In this section, we study the finite-sum problems with the following assumption.

\begin{assumption}\label{assump:tau}
	$f: \Mc \to \mathbb{R}$ is $\tau-$gradient dominated.
\end{assumption}

We denote $\kappa = 2L\tau$ as the condition number. To solve such problems, we propose two algorithms. The first algorithm is shown in Algorithm \ref{algo:gd-1}. It follows the same idea as in \cite{zhang2016riemannian,reddi2016stochastic}. We have the following theorem on its convergence rate.

\begin{algorithm}[ht]
	\caption{R-SPIDER-GD1($x_0, M_0, f$)}\label{algo:gd-1}
	\begin{algorithmic}
		\FOR{$t = 1, \ldots K$}
		\STATE $\epsilon_t = \sqrt{\frac{M_0}{2^t10\tau}  }$
		\STATE $\cS_1 = n, \quad \eta_t = \epsilon/L,\quad q =\ceil{n^{1/2}}$
		\STATE $x_{t} = \textbf{R-SPIDER-nonconvex}(x_{t-1}, \cS_1, q, \eta_t, \epsilon_t, f)$
		\ENDFOR
		\RETURN $x_K$
	\end{algorithmic}

\end{algorithm}

\begin{theorem} \label{thm:finite-gd1}
	Under Assumptions \ref{assump:smooth-xi}, \ref{assump:M}, \ref{assump:tau} and the parameter choice $M_0 \ge f(x_0)-f^*$,  after $T$ iterations, Algorithm \ref{algo:gd-1} returns a solution $x_K$ that satisfies
	$$\mathbb{E}[f(x_K) - f(x^*)] \le 2^{-K}M_0. $$
	Further, we need $\Oc((n+\kappa n^{1/2})\log(1/\epsilon))$ number of IFO calls to achieve $\epsilon$ accuracy in expectation.
	
\end{theorem}
The proof of Theorem \ref{thm:finite-gd1} follows from Theorem \ref{thm:finite-nonconvex} and the gradient dominated property, as shown in Appendix \ref{sec:proof-finite-gd1}.

\setlength{\textfloatsep}{10pt}
\begin{algorithm}[ht]
	\caption{R-SPIDER-GD2($x_0, \cS_1, q, \eta, \epsilon, f$)}\label{algo:gd-2}
	\begin{algorithmic}
		\STATE $\delta \leftarrow \frac{M_0}{4\tau}$
		\FOR{$k = 1, 2, \ldots, qK $}
		\IF{$\mod(k, q) = 0$}
		\STATE draw $n$ samples and evaluate full gradient $v_k \leftarrow \grad f (x_k)$
		\STATE $\delta_k = \delta_{k-1} /2 $
		\ELSE
		\STATE $\delta_k = \delta_{k-1}$
		\STATE draw $\cS_k = \ceil{\min\{n, \frac{qL^2\|Exp^{-1}_{x_{k-1}}(x_k)\|^2}{\delta}\}}$ samples
		\STATE $v_k \leftarrow \grad f_{\cS_k} (x_k) - \Gamma_{x_{k-1}}^{x_k} [ \grad f_{\cS_k} (x_{k-1}) - v_{k-1}]$
		\ENDIF
		\STATE $x_{k+1} \leftarrow \Exp_{x_k} ( - \eta v_k)$
		\ENDFOR
		\RETURN $x_{qK}$
	\end{algorithmic}
\end{algorithm}

The second algorithm, shown in Algorithm \ref{algo:gd-2}, aims to achieve better complexity dependency on $n$. With the following choice of parameters
\begin{align}\label{eq:param-strongly2}
\eta = \frac{1}{2L}, \quad q = \ceil{4L\tau\log(4)},  \quad M_0 \ge M = f(x_0) - f(x^*),
\end{align}
we can make the following statement.

\begin{theorem} \label{thm:finite-gd2}
	Under Assumptions \ref{assump:smooth-xi}, \ref{assump:M}, \ref{assump:tau} and the parameter choice in \eqref{eq:param-strongly2}, after $T = qK = \frac{2KL}{\tau}\log(4)$ iterations, Algorithm \ref{algo:gd-2} returns a solution $x_T$ that satisfies
	$$\mathbb{E}[f(x_T) - f(x^*)] \le 2^{-K}M_0. $$
	Further, the total expected number of IFO calls is $K(n + 25\kappa^2)$. In other word, to achieve $\epsilon$ accuracy, we need $\Oc((n+\kappa^2)\log(1/\epsilon))$ number of IFO calls in expectation.
\end{theorem}

\begin{proof}
	See Appendix \ref{sec:proof-finite-gd2}. The algorithm adaptively choose sample sizes based on the distance of the last update. The expected number of samples queried can be bounded by total sum of squared distances, which can further be bounded by the change in the objective value. 
\end{proof}

In summary, Algorithm~\ref{algo:gd-1} achieves IFO complexity $\Oc((n+\kappa n^{1/2})\log(1/\epsilon))$, while Algorithm~\ref{algo:gd-2} achieves sample complexity $\Oc((n+\kappa^2)\log(1/\epsilon))$. It is unclear to us whether there exists an algorithm that performs uniformly better than both of the proposed algorithms. Further, we wish to point out that, as strong convexity implies gradient dominance, the convergence rates for the above algorithms also apply to  $\frac{1}{2\tau}-$strongly g-convex functions.

%% file: discussion.tex
We introduce Riemannian SPIDER algorithms, a fast variance reduced stochastic gradient algorithm for Riemannian optimization. We analyzed the convergence rates of these algorithms for general smooth geodesically nonconvex functions under both finite-sum and stochastic settings, as well as for gradient dominated functions under the finite-sum setting. We showed that these algorithms improved the best known IFO complexity. We also removed the iteration complexity dependency on the curvature of the manifold.

There are a few open problems. First, the original SPIDER algorithm in \cite{fang2018spider} and Algorithm ~\ref{algo:nonconvex} require very small stepsize. In practice, this usually results in very slow convergence rate. Even though the SPIDER-boost algorithm~\citep{wang2018spiderboost} and Algorithm~\ref{algo:gd-2} utilizes a constant large stepsize, the former one requires random termination of the algorithm, while the latter one requires very large sample size in each iteration. Therefore, none of these algorithms tend to perform well in practice if the implementation follows the theory exactly. Designing and testing practical algorithms with nice theoretical guarantees is left as future work.

Further, we approached the gradient-dominated functions with two different algorithms and got two different convergence rates. We suspect that it is not possible to achieve the best of both worlds at the same time. Proving such a lower bound or finding an algorithm that uniformly dominates both algorithms are both interesting research topics.


%% file: appendix.tex
\begin{center}
  \bfseries\large
  Appendix
\end{center}

\section{Proof of Theorem \ref{thm:stochastic}} \label{sec:proof-stochastic}

The proof of Theorem~\ref{thm:stochastic}, Theorem~\ref{thm:finite-nonconvex} and Theorem~\ref{thm:finite-gd2} all follows from three steps: bound the variance of the gradient estimator; prove that function value decrease in expectation per iteration; bound the number of iterations with Assumption~\ref{assump:M}.

\textit{First} we bound the variance of the estimator $v_k$ at each step.

\begin{lemma}\label{lemma:variance} 
	Under Assumptions\ref{assump:smooth-xi}, \ref{assump:M}, \ref{assump:variance} and parameter setting in \eqref{eq:non-sto-param}, $\forall k \ge 0$, let $k_0 = \floor{k/q}q$. Then the iterates of Algorithm~\ref{algo:nonconvex} satisfy $$\mathbb{E}[\|v_k - \grad f(x_k)\|^2 | \mathcal{F}_{k_0}] \le \epsilon^2$$
\end{lemma}
\begin{proof}
	Let $\F_k$ be the sigma field generated by the random variable $x_k$. Then $\{\F_k\}_{k\ge 0}$ forms a filtration. We can write the following equations
	\begin{align} \label{eq:lemma-var-1}
	&\mathbb{E}[\|v_k - \grad f(x_k)\|^2 | \mathcal{F}_k ] \\
	&= \mathbb{E}[\|\Gamma_{x_{k-1}}^{x_k} [v_{k-1} - \grad f(x_{k-1})]\|^2 | \mathcal{F}_k ] \nonumber \\ \nonumber
	&+ \mathbb{E}[\|\grad f_{\cS_2} (x_k) - \grad f(x_{k}) + \Gamma_{x_{k-1}}^{x_k} [\grad f(x_{k-1}) - \grad f_{\cS_2} (\x_{k-1})]\|^2 | \mathcal{F}_k ] \\
	&+2 \mathbb{E}[\inner{\Gamma_{x_{k-1}}^{x_k} [v_{k-1} - \grad f(x_{k-1})], \grad f_{\cS_2} (\x_k) - \grad f(x_{k}) + \Gamma_{x_{k-1}}^{x_k} [\grad f(x_{k-1}) - \grad f_{\cS_2} (\x_{k-1})]} | \mathcal{F}_k ] \\
	&= \mathbb{E}[\|v_{k-1} - \grad f(x_{k-1})\|^2 | \mathcal{F}_k ] + \nonumber \\
	&\mathbb{E}[\|\grad f_{\cS_2} (\x_k) - \grad f(x_{k}) + \Gamma_{x_{k-1}}^{x_k} [\grad f(x_{k-1}) - \grad f_{\cS_2} (\x_{k-1})]\|^2 | \mathcal{F}_k ].
	\end{align}
	The first equality follows by the identities
	\begin{align*}
		v_k = \grad f_{\cS_2} (x_k) - \Gamma_{x_{k-1}}^{x_k} [ \grad f_{\cS_2} (x_{k-1}) - v_{k-1}], \\
		 \grad f(x_k) = \grad f(x_{k-1}) - \grad f(x_{k-1}) + \grad f(x_k).
	\end{align*}
	The second equality follows by the fact that $ \grad f_{\cS_2} (\x_{k-1})$ and $\grad f_{\cS_2} (\x_k)$ are unbiased estimators.
	Denote $z_i = \grad f(x_k; \xi_i) - \grad f(x_k) - \Gamma_{x_{k-1}}^{x_k}(\grad f(x_{k-1}; \xi_i) - \grad f(x_{k-1})),$ where $f(\cdot; \xi_i)$ is a sampled function from the distribution of $\xi$ as defined in~\eqref{eq:prob-def}. Note that $\mathbb{E}[z_i|\F_{k_0}] = 0$. Hence we have 
	\begin{align*}
	&\mathbb{E}[\|\grad f_{\cS_2} (\x_k) - \grad f(x_{k}) + \Gamma_{x_{k-1}}^{x_k} [\grad f(x_{k-1}) - \grad f_{\cS_2} (\x_{k-1})]\|^2 | \mathcal{F}_k ] = \mathbb{E}[\|\frac{1}{\cS_2}\sum_{i=1}^{\cS_2}z_i\|^2|\F_k] \\
	& = \frac{1}{\cS_2} \mathbb{E}[\|z_i\|^2|\F_k] \\
	& = \frac{1}{\cS_2}\mathbb{E}[\|\grad f(x_k; \xi) - \grad f(x_k) - \Gamma_{x_{k-1}}^{x_k}(\grad f(x_{k-1}; \xi) - \grad f(x_{k-1}))) \|^2|\F_k].
	\end{align*}
	Substitue in \eqref{eq:lemma-var-1} and we get that
	\begin{align*}
	\mathbb{E}[\|v_k - \grad f(x_k)\|^2 | \mathcal{F}_k ]  \le& \mathbb{E}[\|v_{k-1} - \grad f(x_{k-1})\|^2 | \mathcal{F}_k ] \\
	+& \frac{1}{\cS_2}\mathbb{E}[\|\grad f(\x_k;\xi) - \Gamma_{x_{k-1}}^{x_k} [ \grad f (\x_{k-1}; \xi)]\|^2 | \mathcal{F}_k ]\\
	\le& \mathbb{E}[\|v_{k-1} - \grad f(x_{k-1})\|^2 | \mathcal{F}_k ] + L^2\mathbb{E}[\|Exp^{-1}_{x_{k-1}}(x_k)\|^2 | \mathcal{F}_k ]/\cS_2\\
	\le& \mathbb{E}[\|v_{k-1} - \grad f(x_{k-1})\|^2 | \mathcal{F}_k ] + L^2 \eta_{k-1}^2\|v_{k-1}\|^2/\cS_2, \\
	\le& \mathbb{E}[\|v_{k-1} - \grad f(x_{k-1})\|^2 | \mathcal{F}_k ] + \frac{\epsilon^2}{2q}
	\end{align*}
	where the first inequality follows by $\mathbb{E}[\|x\|^2] \ge \mathbb{E}[\|x - \mathbb{E}[x]\|^2]$. The last inequality follows by the value of $\cS_2$. Apply the bound recursively and denote $k_0 = \floor{k/q}q$, we get
	\begin{align*}
	&\mathbb{E}[\|v_k - \grad f(x_k)\|^2 | \mathcal{F}_{k_0}] \le \mathbb{E}[\|v_{k_0} - \grad f(x_{k_0})\|^2 | \mathcal{F}_{k_0}] + q \frac{\epsilon^2}{2q}  \le \epsilon^2.
	\end{align*}
\end{proof}

\textit{Second}, we show that the function value decreases.

\begin{lemma}\label{lemma:decrease} Under Assumptions\ref{assump:smooth-xi}, \ref{assump:M}, \ref{assump:variance} and parameter setting in \eqref{eq:non-sto-param}, $\forall k \ge 0$, let $k_0 = \floor{k/q}q$. Then 
\begin{align*}
	&\mathbb{E}[f(x_{k+1}) - f(x_k) |\mathcal{F}_{k_0}] \le E[- \frac{\|v_k\|^2}{8L} + \frac{1}{4L}  \|\grad f(x_k) - v_k \|^2| \mathcal{F}_{k_0}]
\end{align*}
\end{lemma}
\begin{proof}
	By geodesically $L-$smoothness and \eqref{eq:l-smooth}, we have
	\begin{align*}
	&f(x_{k+1}) - f(x_k) \le \inner{\grad f(x_k), \Exp^{-1}_x(x_{k+1})} + \frac{L}{2}\|\Exp^{-1}_{x_k}(x_{k+1})\|^2 \\
	&\le -\eta_k \inner{\grad f(x_k), v_k } + \frac{L}{2}\eta_k^2\|v_k\|^2 \\
	& = (-\eta_k + \frac{L\eta_k^2}{2})\|v_k\|^2 - \eta_k   \inner{\grad f(x_k) - v_k , v_k }\\
	& \le (-\eta_k/2 + \frac{L\eta_k^2}{2})\|v_k\|^2 + \frac{\eta_k}{2}  \|\grad f(x_k) - v_k \|^2.
	\end{align*}
	The second inequality follows from the update rule of $x_k$ in Algorithm \ref{algo:nonconvex}. The last inequality follows from $-\|a\|^2 - 2\inner{a,b} \le \|b\|^2$. After taking expectation, we get
	\begin{align*}
	&\mathbb{E}[f(x_{k+1}) - f(x_k) |\mathcal{F}_{k_0}] \le \E[\frac{-1}{8L}\|v_k\|^2 + \frac{1}{4L}  \|\grad f(x_k) - v_k \|^2| \mathcal{F}_{k_0}] 
	\end{align*}
	The inequality follows by $\eta_k = \frac{1}{2L}$. 
%

\end{proof}

\textit{Finally} we are ready to prove the theorem.
\begin{proof}[Proof of Theorem \ref{thm:stochastic}]
	
	First, we rearrange and do a telescopic sum of the inequality in Lemma~\ref{lemma:decrease}. 
	
	\begin{align}
	\sum_{k=0}^{T-1}\E[\frac{\|v_k\|^2}{8L} ]  \le \E[f(x_0) - f(x_T) + \sum_{k=0}^{T-1}\frac{1}{4L}  \|\grad f(x_k) - v_k \|^2] 
	\end{align}
	Notice that 
	\begin{align}
		&\frac{1}{T}\sum_{k=0}^{T-1}\E[\|\grad f(x_k)\|^2] \le \frac{1}{T}\sum_{k=0}^{T-1}\E[2\|v_k\|^2 + 2\|v_k-\grad f(x_k)\|^2] \le 16LM/T + 6\epsilon^2. 
	\end{align}
	Substitute in $T = 4ML/\epsilon^2$ and we proved the theorem. The expected number of IFO calls can be computed as 
	\begin{align}
		T\cS_1/q + \E[\sum_{k=0}^{T-1}\cS_{2,k}] \le 8ML\sigma^2/\epsilon^3 + \E[\sum_{k=0}^{T-1}\frac{qL^2\|Exp^{-1}_{x_{k-1}}(x_k)\|^2}{2\epsilon^2}] \\
		\le 8ML\sigma^2/\epsilon^3 + \frac{q}{8\epsilon^2}(8ML + 2T\epsilon^2) \le 8ML(\sigma^2+3)/\epsilon^3 
	\end{align}

\end{proof}

\section{Proof of Theorem \ref{thm:finite-nonconvex}}\label{sec:proof-finite-nonconvex}

The proof techniques are exactly the same as those in Section \ref{sec:proof-stochastic}, with the only changes being that $\mathbb{E}[\|v_{k_0} - \grad f(x_{k_0})\|^2 | \mathcal{F}_{k_0}]\equiv 0$ and we use \eqref{eq:non-finite-param} rather than \eqref{eq:non-sto-param} as the parameters. We state two corresponding lemmas and leave out the details to avoid repetition.
\begin{lemma}\label{lemma:finite-sum-variance} 
	Under Assumptions \ref{assump:smooth-xi}, \ref{assump:M} and the parameter choice in \eqref{eq:non-finite-param}, $\forall k \ge 0$, let $k_0 = \floor{k/q}q$. Then the iterates of Algorithm~\ref{algo:nonconvex} satisfy $$\mathbb{E}[\|v_k - \grad f(x_k)\|^2 | \mathcal{F}_{k_0}] \le \epsilon^2$$
\end{lemma}
\begin{lemma}\label{lemma:decrease} Under Assumptions \ref{assump:smooth-xi}, \ref{assump:M} and the parameter choice in \eqref{eq:non-finite-param}, $\forall k \ge 0$, let $k_0 = \floor{k/q}q$. Then 
	\begin{align*}
	&\mathbb{E}[f(x_{k+1}) - f(x_k) |\mathcal{F}_{k_0}] \le E[- \frac{\|v_k\|^2}{8L} + \frac{1}{4L}  \|\grad f(x_k) - v_k \|^2| \mathcal{F}_{k_0}]
	\end{align*}
\end{lemma}
\begin{proof}[Proof of Theorem \ref{thm:finite-nonconvex}]
	The proof follows exactly the same arguments as the proof of Theorem \ref{thm:stochastic}.
\end{proof}

\section{Proof of Theorem \ref{thm:finite-gd1}}\label{sec:proof-finite-gd1}
\begin{proof}
	By Theorem~\ref{thm:finite-nonconvex}, we know that when the R-SPIDER-nonconvex algorithm terminates in iteration $t$, it returns $x_t$ satisfying 
	$$\E[\|\grad f(x_t)\|^2] \le  \frac{M_0}{2^t\tau}. $$
	By Assumption~\ref{assump:tau}, we know that
	\begin{align} \label{eq:restart}
	\mathbb{E}[f(x_t) - f(x^*)] \le \mathbb{E}[\|\grad f(x_t)\|^2]\tau \le \frac{M_0}{2^t},
	\end{align}
	By theorem \ref{thm:finite-nonconvex}, in iteration $t$ the R-SPIDER-nonconvex algorithm makes less than 
	$$n + \frac{8L(1+\sqrt{n})(f(x_t)- f^*)}{\epsilon_t^2}  = n + \frac{2^t 80 \tau L(1+\sqrt{n})(f(x_t)- f^*)}{M_0}$$ 
	IFO calls in expectation. Take expectation and substitute in the bound in \eqref{eq:restart}, we get that the expected number of IFO calls is less than $(n+ 40\kappa (1+\sqrt{n}))T$.
\end{proof}

\section{Proof of Theorem \ref{thm:finite-gd2}}\label{sec:proof-finite-gd2}

As usual, we start with bounding the variance of the gradient estimator.
\begin{lemma}\label{lemma:gd2-variance} Let $k_0 = \floor{k/q}q$.
	Under Assumptions \ref{assump:smooth-xi}, \ref{assump:M}, \ref{assump:tau} and the parameter choice in \eqref{eq:param-strongly2}, let $v_k, x_k$ be intermediate values of Algorithm\ref{algo:gd-2}. $\forall k \ge 0$, $\mathbb{E}[\|v_k - \grad f(x_k)\|^2 | \mathcal{F}_{k_0}] \le \delta_{k_0}.$
\end{lemma}
\begin{proof}
	Based on the implementation of Algorithm \ref{algo:gd-2}, for $i \in \{k_0, k_0+1, ..., k_0+q-1\}$, notice that $\delta_i = \delta_{k_0}$. We denote $\delta = \delta_i = \delta_{k_0}$ for simplicity. Following the proof procedure in Lemma \ref{lemma:variance},
	\begin{align}
	&\mathbb{E}[\|v_k - \grad f(x_k)\|^2 | \mathcal{F}_k ]  \le \mathbb{E}[\|v_{k-1} - \grad f(x_{k-1})\|^2 | \mathcal{F}_k ] + L^2\|Exp^{-1}_{x_{k-1}}(x_k)\|^2/\cS_k \\
	& \le  \mathbb{E}[\|v_{k-1} - \grad f(x_{k-1})\|^2 | \mathcal{F}_k ] + \delta/q.
	\end{align}
	Apply this recursively and denote $k_0 = \floor{k/q}q$, we get
	\begin{align}
	&\mathbb{E}[\|v_k - \grad f(x_k)\|^2 | \mathcal{F}_{k_0}] \le \mathbb{E}[\|v_{k_0} - \grad f(x_{k_0})\|^2 | \mathcal{F}_{k_0}] + q\delta/q  \le \delta.
	\end{align}
\end{proof}

Then we prove that function value decreases in each epoch ($q$ iterations).

\begin{lemma}\label{lemma:gd2-decrease} Let $k_0 = \floor{k/q}q$, $\forall k \in \mathbb{Z}_{>0}$. Under Assumptions \ref{assump:smooth-xi}, \ref{assump:M}, \ref{assump:tau} and the parameter choice in \eqref{eq:param-strongly2}, if we run Algorithm \ref{algo:gd-2} for $q$ iteration, we have
	\begin{align}
	\mathbb{E}[f(x_{k_0+q}) - f(x^*)|\mathcal{F}_{k_0}] \le \max\{(f(x_{k_0})- f(x^*))/2, \delta_{k_0}\tau\}.
	\end{align}
	
\end{lemma}
\begin{proof}
	By geodesically $L-$smooth, we have 
	\begin{align}
	&f(x_{k+1}) - f(x_k) \le \inner{\grad f(x_k), \Exp^{-1}_{x_k}(x_{k+1})} + \frac{L}{2}\|\Exp^{-1}_{x_k}(x_{k+1})\|^2 \\
	& = - \eta \inner{\grad f(x_k), v_k } + \frac{L\eta^2}{2}\|v_k\|^2 \\
	& = -  \frac{1}{2L} \inner{\grad f(x_k), v_k } + \frac{1}{8L}\|v_k\|^2  .
	\end{align}
	After taking expectation, we get
	\begin{align}
	&\mathbb{E}[f(x_{k+1}) - f(x_k) |\mathcal{F}_{k_0}] \le \E[- \frac{1}{4L} \|\grad f(x_k)\|^2|\mathcal{F}_{k_0}] + \frac{1}{4L} \mathbb{E}[ \|v_k - \grad f(x_k)\|^2 |\mathcal{F}_{k_0}].
	\end{align}
	We used the fact that $$\E[\|v_k\|^2|\mathcal{F}_{k_0}] = \mathbb{E}[ \|v_k - \grad f(x_k)\|^2 |\mathcal{F}_{k_0}] + \E[\|\grad f(x_k)\|^2|\mathcal{F}_{k_0}]+ \mathbb{E}[ 2\inner{v_k - \grad f(x_k), \grad f(x_k)} |\mathcal{F}_{k_0}].$$
		Let $\Delta_k := f(x_k) - f(x^*)$. By Assumption \ref{assump:tau} and Lemma \ref{lemma:gd2-variance}, we get
	\begin{align}
	&\mathbb{E}[\Delta_{k+1} |\mathcal{F}_{k_0}] \le \E[(1-\frac{1}{4L\tau})\Delta_k + \frac{1}{4L}\delta_{k_0} |\mathcal{F}_{k_0}] .
	\end{align}
	By choice of parameter defined in \ref{eq:param-strongly2}, we get for $\rho = 1-\frac{1}{4L\tau}$
	\begin{align}
	\mathbb{E}[\Delta_{k+1} |\mathcal{F}_{k}] \le \E[\rho\Delta_k + \frac{1}{4L}\delta_{k_0} |\mathcal{F}_{k_0}].
	\end{align}

	Since $\Delta_{k_0} = f(x_{k_0}) - f(x^*)$, we know that 
	after $q = \ceil{4L\tau\log(4)}$ iterations, we have
	\begin{align}
	\mathbb{E}[\Delta_{k_0+q} |\mathcal{F}_{k_0}] \le \Delta_{k_0}/4 + \delta_{k_0}\tau\log(4) \le \max\{\Delta_{k_0}/2, 2\delta_{k_0}\tau\}.
	\end{align}
\end{proof}

Now we are ready to prove Theorem~\ref{thm:finite-gd2}. 

\begin{proof}[Proof of Theorem \ref{thm:finite-gd2}]

We prove the convergence result by induction. The base case follows by choosing $M_0$ such that
\begin{align}
M_0 \ge M = f(x_0) - f(x^*).
\end{align}
If $\mathbb{E}[f(x_{qn}) - f(x^*)] \le 2^{-n}M_0$, then by Lemma \ref{lemma:gd2-decrease}, we have
\begin{align}
\mathbb{E}[f(x_{qn+q}) - f(x^*)|\F_{qn}] \le \max\{2\delta_{qn}\tau, (f(x_{qn}) - f(x^*))/2\}.
\end{align}
By Algorithm \ref{algo:gd-2}, we know that $\delta_{qn} = 2^{-n-2} M_0 / \tau $. After taking expectation, we get by inductive assumption that
\begin{align} \label{eq:converge}
\mathbb{E}[f(x_{qn+q}) - f(x^*)] \le 2\delta_{qn}\tau = 2^{-(n+1)}M_0.
\end{align}

We have already proved the convergence of objective values. Below, we bound the expected number of IFO calls in total. By $L-$smoothness, we have
\begin{align}
&f(x_{k+1}) - f(x_k) \le \inner{\grad f(x_k), \Exp^{-1}_{x_k}(x_{k+1})} + \frac{L}{2}\|\Exp^{-1}_{x_k}(x_{k+1})\|^2 \\
&\le - \eta \inner{\grad f(x_k), v_k } + \frac{L\eta^2}{2}\|v_k\|^2 \\
&\le \frac{1}{4L} \inner{ 2(v_k - \grad f(x_k)), v_k } - \frac{3}{8L}\|v_k\|^2\\
& \le -\frac{1}{8L}\|v_k\|^2 + \frac{1}{L}\|v_k -  \grad f(x_k)\|^2 .
\end{align}

In the last inequality we used the fact that $\inner{a,b}\le \|a\|^2 + \|b\|^2$. Take expectation and we get
\begin{align}
&\mathbb{E}[f(x_{k+1}) - f(x_k)|\F_{k_0}] \le - \frac{L}{2}\mathbb{E}[\|\Exp^{-1}_{x_k}(x_{k+1})\|^2|\F_{k_0}] + \delta_{k_0}/L.
\end{align}
Let $k_0 = nq$ be a integer multiple of q. Substitute in the choice of $\eta, \delta$ and sum over $k$. Then we get
\begin{align}
&\mathbb{E}[f(x_{k_0+q}) - f(x_{k_0})|\F_{k_0}] \le  (- \frac{L}{2})\sum_{k=k_0}^{k_0+q-1}\mathbb{E}[\|\Exp^{-1}_{x_k}(x_{k+1})\|^2|\F_{k_0}] + q\delta_{k_0}/L .
\end{align}
Rearrange and we get
\begin{align}
&\sum_{k=k_0}^{k+q-1}\mathbb{E}[\|\Exp^{-1}_{x_k}(x_{k+1})\|^2|\F_k] \le \frac{2}{L}( \mathbb{E}[f(x_{k_0}) - f(x_{k_0+q})|\F_{k_0}] + q\delta_{k_0}/L ).
\end{align}
Given that $f(x_{k_0+1}) \ge f(x^*)$, we get
\begin{align}
&\mathbb{E}[\sum_{k=k_0}^{k_0+q-1} \cS_k|\F_{k_0}] = \frac{qL^2}{\delta_{k_0}}\sum_{k=k_0}^{k_0+q-1}\mathbb{E}[\|\Exp^{-1}_{x_k}(x_{k+1})\|^2|\F_{k_0}] + n \\
&\le 2qL ( (f(x_{k_0}) - f(x^*))/\delta_{k_0} + q/L ) + n.
\end{align}

Take another expectation with respect to $x_{k_0}$ 
\begin{align}
\mathbb{E}[\sum_{k=k_0+1}^{k+q} \cS_k] \le 2qL( \tau + q/L) + n \le 100L^2\tau^2 + n.
\end{align}
Therefore, the total expected number of IFO calls is limited by $K(n + 100L^2\tau^2)$. 
\end{proof}